\documentclass[11pt]{article}
\usepackage{amsmath,amsfonts,amsthm,amssymb,color}

\newtheorem{theorem}{Theorem}[section]
\newtheorem{corollary}[theorem]{Corollary}
\newtheorem{lemma}[theorem]{Lemma}

\theoremstyle{definition}
\newtheorem{definition}[theorem]{Definition}

\newdimen\pIR
\newcommand\StevesR{{\rm I\kern\pIR R}}

\def\prob#1#2{\mbox{Pr}_{#1}\left[ #2 \right]}

\def\expec#1#2{{\mathbb{E}}_{#1}\left[ #2 \right]}

\def\defeq{\stackrel{\mathrm{def}}{=}}
\def\setof#1{\left\{#1  \right\}}
\def\sizeof#1{\left|#1  \right|}

\newcommand{\sym}[1]{\mathrm{sym} (#1)}

\newcommand{\R}{\mathbb{R}}

\newcommand{\mydet}[1]{\det\left(#1\right)}

\begin{document}

\title{
Interlacing Families I: Bipartite Ramanujan Graphs of All Degrees
\thanks{
A preliminary version of this paper appeared in the Proceedings of the
54th IEEE Annual Symposium on Foundations of Computer Science.
}}

\author{
Adam W. Marcus\\
Yale University\\
\and
Daniel A. Spielman \\ 
Yale University
\and 
Nikhil Srivastava\\
Microsoft Research, 
India
}

\maketitle

\begin{abstract}
We prove that there exist infinite families of regular bipartite Ramanujan graphs
  of every degree bigger than $2$.
We do this by proving a variant of a conjecture of Bilu and Linial
  about the existence of good $2$-lifts of every graph.

We also establish the existence of infinite families of `irregular Ramanujan' graphs, whose
  eigenvalues are bounded by the spectral radius of their universal cover. Such
  families were conjectured to exist by Linial and others.
In particular, we prove the existence of infinite families of $(c,d)$-biregular bipartite graphs
  with all non-trivial eigenvalues bounded by $\sqrt{c-1}+\sqrt{d-1}$, for all $c, d \geq 3$.

Our proof exploits a new technique for demonstrating the existence
  of useful combinatorial objects that we call the ``method of interlacing polynomials''.
\end{abstract}

\section{Introduction}\label{sec:intro}

Ramanujan graphs have been the focus of substantial study in Theoretical Computer Science
  and Mathematics.
They are graphs whose non-trivial adjacency matrix eigenvalues are as small as possible.
Previous constructions of Ramanujan graphs have been sporadic, only producing Ramanujan
  graphs of particular degrees.
In this paper, we prove a variant of a conjecture of Bilu and Linial~\cite{BiluLinial},
  and use it to realize an approach they suggested for constructing bipartite
  Ramanujan graphs of every degree.

Our main technical contribution is a novel existence argument.
The conjecture of Bilu and Linial requires us to prove that every 
   graph has a signed adjacency matrix with all of its eigenvalues in a small range.
We do this by proving that the roots of the {\em expected} characteristic polynomial
  of a randomly signed adjacency matrix lie in this range.
In general, a statement like this is useless, as the roots of a sum of polynomials do not necessarily
  have anything to do with the roots of the polynomials in the sum.
However, there seem to be many sums of combinatorial polynomials for which this intuition is wrong.
With this in mind, we identify certain special collections of polynomials which we call ``interlacing families'', and 
  prove that such families always contain a polynomial whose largest root is at most
  the largest root of the sum.
We show that the polynomials arising from signings of a graph form such a family.
To finish the proof, we then bound the largest root of the sum of the characteristic 
  polynomials of the signed adjacency matrices of a graph by observing that this sum 
  is the well-studied matching polynomial of the graph.

This paper is the first one in a series which develops
  the method of interlacing polynomials.
In the next paper \cite{IF2}, we use the method 
  to give a positive resolution to the Kadison--Singer problem. 

%Unlike previous approaches, our techniques do not rely on regularity and 
%  also establish the existence of `irregular' bipartite Ramanujan graphs, which
%  will be defined precisely in the next section.

\section{Technical Introduction and Preliminaries}
\subsection{Ramanujan Graphs}
Ramanujan graphs are defined in terms of the eigenvalues of their adjacency matrices.
If $G$ is a $d$-regular graph and $A$ is its adjacency matrix, then $d$ is always
  an eigenvalue of $A$.
The matrix $A$ has an eigenvalue of $-d$ if and only if $G$ is bipartite.
The eigenvalues of $d$, and $-d$ when $G$ is bipartite, are called the \textit{trivial} 
  eigenvalues of $A$.
Following Lubotzky, Phillips and Sarnak~\cite{LPS}, we say that a $d$-regular
graph is \textit{Ramanujan} if all of its non-trivial eigenvalues
  lie between $-2\sqrt{d-1}$ and $2 \sqrt{d-1}$.
It is easy to construct Ramanujan graphs with a small number of vertices:
  $d$-regular complete graphs and complete bipartite graphs are Ramanujan.
The challenge is to construct an infinite family of $d$-regular graphs that
  are all Ramanujan.
One cannot construct infinite families of $d$-regular graphs whose eigenvalues
   lie in a smaller range:
  the Alon--Boppana bound (see \cite{nilli}) tells us that
   for every constant $\epsilon > 0$,
   every sufficiently large $d$-regular graph has a 
   non-trivial eigenvalue with absolute value at least
   $2\sqrt{d-1}-\epsilon$.

%The Alon--Boppana bound implies that such a family would be asymptotically optimal in the sense that no infinite families of $d$-regular graphs can have all trivial eigenvalues in a smaller range:
%\begin{lemma}[Alon--Boppana (see \cite{nilli})] \label{AlonBoppana}
%For all positive integers $d$ and all $\epsilon > 0$ there exists an $n = n(\epsilon, d)$ such that every $d$-regular graph with at least $n$ vertices contains a non-trivial eigenvalue $\lambda$ such that $\lambda > 2 \sqrt{d-1} - \epsilon$.
%\end{lemma}

Lubotzky, Phillips and Sarnak~\cite{LPS} and Margulis~\cite{Margulis}
  were the first to construct 
  infinite families of
  Ramanujan graphs of constant degree.
They built both bipartite and non-bipartite Ramanujan graphs from Cayley graphs.
All of their graphs are regular and have degrees $p+1$ where $p$ is a prime.
There have been very few other constructions of Ramanujan graphs
  \cite{pizer,chiu1992cubic,jordan1997ramanujan,Morgenstern}.
To the best of our knowledge, the only degrees for which infinite families of Ramanujan
  graphs were previously known to exist were those of the form $q+1$ where $q$ is a 
  prime power.
Lubotzky~\cite[Problem 10.7.3]{LubotzkyBook} asked whether there exist infinite
  families of Ramanujan graphs of every degree greater than 2.
We resolve this conjecture in the affirmative in the bipartite case.

\subsection{2-Lifts}\label{intro:lifts}

Bilu and Linial~\cite{BiluLinial} suggested constructing Ramanujan graphs through
  a sequence of 2-lifts of a base graph.
Given a graph $G = (V,E)$, a 2-lift of $G$ is a graph that has two vertices
  for each vertex in $V$.
This pair of vertices is called the \textit{fibre} of the original vertex.
Every edge in $E$ corresponds to two edges in the 2-lift.
If $(u,v)$ is an edge in $E$, $\setof{u_{0}, u_{1}}$ is the fibre of $u$,
  and $\setof{v_{0}, v_{1}}$ is the fibre of $v$, then the 2-lift can
  either contain the pair of edges
\begin{align}
&  \setof{(u_{0}, v_{0}), (u_{1}, v_{1})}, \text{or} \label{posedge}\\
&  \setof{(u_{0}, v_{1}), (u_{1}, v_{0})}. \label{negedge}
\end{align}
If only edge pairs of the first type appear, then the 2-lift is just
  two disjoint copies of the original graph.
If only edge pairs of the second type appear, then we obtain the {\em double-cover} of $G$.

To analyze the eigenvalues of a 2-lift, Bilu and Linial study 
  \textit{signings} $s : E \rightarrow \setof{\pm 1}$ of the edges of $G$.
They place signings in one-to-one correspondence with 2-lifts
   by setting $s (u,v) = 1$ if edges of type \eqref{posedge}
  appear in the 2-lift, and $s (u,v) = -1$ if edges of type \eqref{negedge} appear.
They then define the signed adjacency matrix $A_s$ to be the same as the adjacency matrix of $G$,
  except that the entries corresponding to an edge
  $(u,v)$ are $s (u,v)$.
They prove~\cite[Lemma~3.1]{BiluLinial}
  that the eigenvalues of the 2-lift are the union, taken with multiplicity,
  of the eigenvalues of the adjacency matrix $A$ and those of the signed
  adjacency matrix $A_s$.
Following Friedman~\cite{friedman2003relative}, they refer to the eigenvalues of $A$
  as the \textit{old eigenvalues} and the eigenvalues of $A_s$
  as the \textit{new eigenvalues}.
The main result of their paper is that every graph of maximal degree $d$ has a signing in which
  all of the new eigenvalues have absolute value at most
  $O (\sqrt{d \log^{3} d} )$.
They then build arbitrarily large $d$-regular 
  expander graphs by repeatedly taking 2-lifts of
  a complete graph on $d+1$
  vertices.

Bilu and Linial conjectured that every $d$-regular graph
  has a signing
 in which
  all of the new eigenvalues have absolute value at most
  $2 \sqrt{d -1}$.
If one repleatedly applied the corresponding 2-lifts to the $d$-regular
  complete graph, one would obtain an infinite sequence of
  $d$-regular Ramanujan graphs.
We prove a weak version of Bilu and Linial's conjecture:
  every $d$-regular graph has a signing
 in which
  all of the new eigenvalues are at most
  $2 \sqrt{d -1}$.
The difference between our result and the original conjecture
  is that we do not control the smallest new eigenvalue.
This is why we consider bipartite graphs.
The eigenvalues of the adjacency matrices of bipartite graphs are
  symmetric about zero (see, for example, \cite[Theorem 2.4.2]{GodsilBook}).
So, a bound on the smallest non-trivial eigenvalue follows from
  a bound on the largest.
We also use the fact that a 2-lift of a bipartite graph
   is also bipartite.
By repeatedly applying the corresponding 2-lifts to the $d$-regular
  complete bipartite graph, we obtain an infinite sequence
  of $d$-regular bipartite Ramanujan graphs.

\subsection{Irregular Ramanujan Graphs and Universal Covers}\label{intro:irreg}
We say that a bipartite graph is $(c,d)$-biregular if all vertices on one
  side of the bipartition have degree $c$ and all vertices on the other side
  have degree $d$.
The adjacency matrix of a $(c,d)$-biregular graph always has eigenvalues
  $\pm \sqrt{cd}$; these are its trivial eigenvalues.
Feng and Li~\cite{FengLi} (see also~\cite{LiSole}) prove a generalization of the Alon--Boppana bound that applies
  to $(c,d)$-biregular graphs:
 for all $\epsilon > 0$, all sufficiently large $(c,d)$-biregular graphs
  have a non-trivial eigenvalue that is at least $\sqrt{c-1} + \sqrt{d-1} - \epsilon$.
Thus, we say that a $(c,d)$-biregular graph is Ramanujan if all of its non-trivial
  eigenvalues have absolute value at most
$
  \sqrt{c-1} + \sqrt{d-1}.
$
We prove the existence of infinite families of $(c,d)$-biregular Ramanujan graphs for all $c, d \geq 3$.

The regular and biregular Ramanujan graphs discussed above are actually
  special cases of a more general phenomenon.
To describe it, we will require a construction known as the {\em universal cover}.
The universal cover of a graph $G$ is the infinite tree $T$ such that every
  connected lift of $G$ is a quotient of the tree (see, e.g., \cite[Section
  6]{hoory2006expander}).
It can be defined concretely by first fixing a 
  ``root'' vertex $v_{0}\in G$, and then placing one vertex in $T$ for every
  non-backtracking walk $(v_0,v_1,\ldots,v_\ell)$ of any length $\ell\in\mathbb{N}$
  starting at $v_0$, where a walk is non-backtracking if $v_{i-1}\neq v_{i+1}$
  for all $i$.
Two vertices of $T$ are adjacent if and only if 
  the walk corresponding to one can be obtained by appending one vertex to
  the walk corresponding to the other.
That is, the edges of $T$ are all of the form $(v_0,v_1,\ldots, v_\ell)\sim(v_0,v_1,\ldots,v_\ell,v_{\ell+1})$.
The universal cover of a graph is unique up to isomorphism, independent of the
  choice of $v_0$.

The adjacency matrix $A_T$ of the universal cover $T$ is an infinite-dimensional
  symmetric matrix.
We will be interested in the {\em spectral radius} $\rho(T)$ of $T$, which may
  be defined\footnote{In functional analysis, the spectral radius of an
  infinite-dimensional operator $A$ is traditionally defined to be the largest $\lambda$ for which 
  $(A-\lambda I)$ is unbounded. However, in the case of self-adjoint operators,
  this definition is equivalent to the one presented here 
  (see, for example, Theorem~VI.6 in \cite{FuncAnal}).} as:
  \begin{equation}\label{specrad}
  \rho(T) := \sup_{\|x\|_2=1}{\|A_Tx\|_2}
  \end{equation}
where $\|x\|_2^2:=\sum_{i=1}^\infty x(i)^2$ whenever the series converges.
Naturally, the spectral radius of a finite tree is defined to be the norm of its adjacency
  matrix.

With these notions in hand, we can state the definition of an irregular
  Ramanujan graph.
As before, the largest (and smallest, in the bipartite case) eigenvalues
  of finite adjacency matrices are considered trivial.
Greenberg \cite{greenberg} (see also~\cite{cioabua2006eigenvalues})
  showed that for every $\epsilon > 0$ and every
  infinite family of graphs that have the same universal cover $T$,
  all sufficiently large graphs in the family have a non-trivial eigenvalue
  that is at least $\rho (T) - \epsilon$.
Following Hoory, Linial, and Wigderson \cite[Definition
  6.7]{hoory2006expander}, we therefore define an arbitrary graph to be Ramanujan if all of its
  non-trivial eigenvalues are smaller in absolute value than the spectral 
  radius of its universal cover.

The universal cover of every $d$-regular graph is the infinite $d$-ary tree, whereas the universal 
  cover of every $(c,d)$-biregular graph is the infinite $(c,d)-$biregular tree in which the degrees alternate between $c$ and $d$ on every other level~\cite{LiSole}.
The former tree is known to have spectral radius $2 \sqrt{d-1}$ while the latter has
  a spectral radius of $\sqrt{c-1} + \sqrt{d-1}$ (see~\cite{godsilMohar,LiSole}).
Thus, a definition based on universal covers generalizes both the regular and
  biregular definitions of Ramanujan graphs, and the bound of 
  Greenberg generalizes both the Alon-Boppana and Feng-Li bounds.

In this general setting, we show that every graph $G$ has a 2-lift in which all of the new eigenvalues
  are less than the spectral radius of its universal cover.
%By applying these 2-lifts to the $(c,d)$-biregular complete bipartite graph,
%  we obtain an infinite family of $(c,d)$-biregular Ramanujan graphs.
Applying these 2-lifts inductively to any finite irregular bipartite Ramanujan graph yields an
  infinite family of irregular bipartite Ramanujan graphs whose degree distribution
  matches that of the initial graph (since taking a 2-lift simply doubles the
  number of vertices of each degree).
In particular, applying them to the
  $(c,d)$-biregular complete bipartite graph yields an infinite family of
  $(c,d)$-biregular Ramanujan graphs.
As far as we know, infinite families of irregular Ramanujan graphs were not
  known to exist prior to this work.

%Finally, we mention that, in general, infinite linear operators can exhibit behavior that is not seen in finite linear algebra.
%We do state one result in its full generality, since it plays a large role in the proof.
%We will be interested in the spectral properties of the universal cover, forcing us into the realm of functional analysis.  
%In this paper, we will try to avoid much of the functional analysis that comes from dealing with infinite trees by merely quoting the applicable results.
%For the most part, the reader can treat the adjacency operator of an infinite tree to be the limit of the adjacency matrices of increasingly large finite subtrees (we will discuss the validity of this assertion in Section~\ref{sec:related}, but the details are not necessary for the main results).

\subsection{Related Work} \label{sec:related}

There have been numerous studies of random lifts of graphs.
For some results on the spectra of random lifts, we point the reader to
  \cite{amit2002random,linial2005random,amit2006random,linial2010word,addario2010spectrum,lubetzky2011spectra}.
Friedman \cite{friedman2008} has proved that almost every $d$-regular graph almost meets the
  Ramanujan bound: he shows that for every $\epsilon > 0$
  the absolute value of all the non-trivial eigenvalues of
   almost every sufficiently large
  $d$-regular graph are at most $2 \sqrt{d-1} + \epsilon$.
In the irregular case, Puder~\cite{puder2012expansion}
  has shown that with high probability a high-order lift of a graph $G$ 
  has new eigenvalues that are bounded in absolute value by $\sqrt{3}\rho$,
  where $\rho$ is the spectral radius of the universal cover of $G$.

We remark that constructing bipartite Ramanujan graphs is at least as easy
  as constructing non-bipartite ones:
  the double-cover of a $d$-regular non-bipartite Ramanujan graph
  is a $d$-regular bipartite Ramanujan graph.
For many applications of expander graphs, we refer the reader to
  \cite{hoory2006expander}.
For those applications of expanders that just require upper bounds on the
  second eigenvalue, one can use bipartite Ramanujan graphs.
Some applications actually require bipartite expanders, while
  others require the non-bipartite ones.
For example, the explicit constructions of error correcting codes of
  Sipser and Spielman~\cite{sipserSpielman} require non-bipartite expanders,
  while the improvements of their construction 
 \cite{zemor2001expander,roth2006improved,ashikhmin2006decoding}
  require bipartite Ramanujan expanders.

\section{2-Lifts and The Matching Polynomial}\label{sec:matching}

For a graph $G$, let $m_{i}$ denote the number of matchings in $G$ with $i$ edges.
Set $m_{0} = 1$.
Heilmann and Lieb~\cite{HeilmannLieb} defined the \textit{matching polynomial} of $G$
  to be the polynomial
\[
  \mu_{G} (x) \defeq \sum_{i\geq 0} x^{n-2i} (-1)^{i} m_{i},
\]
where $n$ is the number of vertices in the graph.
They proved two remarkable theorems about the matching polynomial that we will exploit
  in this paper.
It is worth mentioning that the proofs of these theorems are elementary and short, relying
  only on simple recurrence formulas for the matching polynomial.

\begin{theorem}[Theorem 4.2 in \cite{HeilmannLieb}]\label{thm:matchingReal}
For every graph $G$, $\mu_{G} (x)$ has only real roots.
\end{theorem}

%\nikhil{I changed this to max degree instead of $d$-regular. Also in Sec.5.}
\begin{theorem}[Theorem 4.3 in \cite{HeilmannLieb}]\label{thm:matchingBound}
For every graph $G$ of maximum degree $d$, all of the roots of $\mu_{G} (x)$
  have absolute value at most $2 \sqrt{d-1}$.
\end{theorem}

The preceding theorems will allow us to prove the existence of infinite families of
  $d$-regular bipartite Ramanujan graphs.
To handle the irregular case, we will require a refinement of these results due to Godsil.
This refinement uses the concept of
  a {\em path tree}, which was also introduced by Godsil (see  \cite{Godsil} or
  \cite[Section 6]{GodsilBook}).
Recall that a path in $G$ is a walk that does not visit any vertex twice.

\begin{definition}\label{def:pathtree} Given a graph $G$ and a vertex $u$,
  the path tree $P (G,u)$ contains one vertex for every path in $G$ 
  (with distinct vertices)
  that starts at $u$.
Two paths are adjacent if one can be
  obtained by appending one vertex to the other.
That is,
 all edges of $P(G,u)$ are all of the form $(u,v_1,\ldots,
  v_\ell)\sim(u,v_1,\ldots,v_\ell,v_{\ell+1})$.
\end{definition}

The path tree provides a natural relationship between the roots of the matching
  polynomial of a graph and the spectral radius of its universal cover:
\begin{theorem}[Godsil \cite{Godsil}] \label{thm:godsil} Let $P (G,u)$ be a
path tree of $G$. Then the matching polynomial of $G$ divides the characteristic polynomial
of the adjacency matrix of $P (G,u)$. In particular, all of the roots of
$\mu_G(x)$ are real and have absolute value at most $\rho(P (G,u))$.
\end{theorem}
\begin{lemma}\label{lem:coverBound}
Let $G$ be a graph and let $T$ be its universal cover. 
Then the roots of $\mu_G(x)$ are bounded in absolute value
  by $\rho(T)$.
\end{lemma}
\begin{proof} Let $u$ be any vertex of $G$ and let $P$ be the path tree rooted
at $u$. Since the paths that correspond to the vertices of $P$ are themselves non-backtracking walks
  (as defined in Section~\ref{intro:irreg}), $P$ is a finite induced subgraph of
  the universal cover $T$, and  $A_P$ is a finite submatrix of
$A_T$. By Theorem~\ref{thm:godsil}, the roots of $\mu_G$ are bounded by
\begin{align*}
\|A_{P}\|_{2} &= \sup_{\|x\|_{2}=1}\|A_{P}x\|_{2} 
\\&\le \sup_{\|y\|_{2}=1,\mathrm{supp}(y)\subset P}\|A_Ty\|_{2} 
\\&\le \sup_{\|y\|_{2}=1}\|A_Ty\|_{2} = \rho(T),
\end{align*}
as desired.
\end{proof}

We remark that one can directly prove
  an upper bound of $2\sqrt{d-1}$ on the spectral radius of a path tree of a $d$-regular graph
  and an upper bound of $\sqrt{c-1}+\sqrt{d-1}$ on the spectral radius of a path tree of a  $(c,d)$-regular bipartite 
  graph without considering infinite trees.
We point the reader to Section 5.6 of Godsil's book \cite{GodsilBook} for an elementary argument.

We now recall an identity of Godsil and Gutman: the
  expected characteristic polynomial of a random signing of the adjacency matrix
  of a graph is equal to its matching polynomial.
To associate a signing of the edges of $G$ with a vector in $\setof{\pm 1}^{m}$,
  we choose an arbitrary ordering of the $m$ edges of $G$,
  denote the edges by $e_{1}, \dotsc , e_{m}$, and denote a signing of
  these edges by $s \in \setof{\pm 1}^{m}$.
We then let $A_{s}$ denote the signed adjacency matrix corresponding to $s$, 
  and define $f_{s} (x) =\mydet{xI-A_s}$ to be characteristic polynomial of $A_{s}$.
%Let $e_{1}, \dots, e_{m}$ be the edges of $G$ (ordered arbitrarily).
%We will use $s \in \setof{\pm 1}^{m}$ to denote a signing of the edges of $G$, as discussed in Section~\ref{intro:lifts}.
%We use $A_{s}$ to denote the signed adjacency matrix corresponding to a given signing $s$ and
%  $f_{s} (x)$ to denote the characteristic polynomial of $A_{s}$.
%The following theorem shows the relationship between the matching polynomial and
%the collection $\{ f_s \}_{s\in\{\pm\}^m}$.

\begin{theorem}[Corollary 2.2 of Godsil and Gutman~\cite{GodsilGutman}]\label{thm:matching}
\[
\expec{s \in \setof{\pm 1}^{m}}{f_{s} (x)}
 = 
 \mu_{G} (x).
\]
\end{theorem}

For the convenience of the reader, we present a simple proof of this theorem
  in Appendix~\ref{sec:ggtheorem}.

%As mentioned in the introduction, knowing the roots of the sum of polynomials 
%  (in general) tells nothing about the roots of the original polynomials.
To prove that a good lift exists, it suffices, by Theorems
\ref{thm:matchingBound} and \ref{thm:matching}, to show that there is a signing $s$ so that the largest root of
  $f_{s} (x)$ is at most the largest root of $\expec{s \in \setof{\pm
1}^{m}}{f_{s} (x)}$.
To do this, we prove that the polynomials $\{f_{s} (x)\}_{s\in\{\pm 1\}^m}$
  are what we call an interlacing family.
We define interlacing families and examine their properties in the next section.
%We prove that the polynomials $p_{s} (x)$ are an interlacing family in Section~\ref{sec:rr}. 

\section{Interlacing Families}\label{sec:interlacing}

%\dan{It might be clearer if the put the interlacing polynomial inside, giving it degree $d-1$
%  and the polynomial we care about degree $d$.}\\
%\nikhil{Already done. It might be even clearer if we point out that there is
%nothing special about the common interlacing polynomial, other than the fact
%that it separates the root zones of the polynomials we are interested in.}

\begin{definition}\label{def:interlacing}
We say that a polynomial $g(x) = \prod_{i=1}^{n-1} (x - \alpha_{i})$ \emph{interlaces} a polynomial 
  $f(x) = \prod_{i=1}^{n} (x - \beta_{i})$ if
\[
  \beta_{1} \leq \alpha_{1} \leq \beta_{2} \leq \alpha_{2} \leq \dotsb \leq
  \alpha_{n-1}\leq \beta_{n} 
\]
We say that polynomials $f_{1}, \dotsc , f_{k}$ have a \emph{common interlacing}
  if there is a polynomial $g$ so that $g$ interlaces $f_i$ for each $i$.
%We remark that this is equivalent to saying that there are closed intervals
%  $\Lambda_1,\ldots,\Lambda_n\subset\mathbb{R}$ such that for every $j=1,\ldots,n$ and
%  $i=1,\ldots,k$, the $j^$th largest root of $f_i$ is contained in $\Lambda_j$,
%  and $\sup \Lambda_i\le \inf \Lambda_j$ whenever $i<j$.
\end{definition}

Let $\beta_{i,j}$ be the $j$th smallest root of $f_{i}$.
The polynomials $f_{1}, \dotsc , f_{k}$ have a common interlacing if 
  and only if there are
  numbers $\alpha_{0} \leq \alpha_{1} \leq  \dotsb \leq \alpha_{n}$ so that
  $\beta_{i,j} \in [\alpha_{j-1}, \alpha_{j}]$ for all $i$ and $j$.
The numbers $\alpha_{1}, \dotsc , \alpha_{n-1}$ come from the roots of the polynomial $g$,
  and $\alpha_{0}$ ($\alpha_{n}$) can be chosen to be any number that is smaller (larger) than
  all of the roots of all of the $f_{i}$.

\begin{lemma}\label{lem:interlacing}
Let $f_{1}, \dotsc , f_{k}$ be polynomials of the same degree that are real-rooted and have positive leading coefficients.
Define
\[
  f_{\emptyset} = \sum_{i=1}^{k} f_{i}.
\]
If $f_{1}, \dotsc , f_{k}$ have a common interlacing,
  then
% $q_{\emptyset}$ is also real-rooted and
  there exists an $i$ so that the largest root of $f_{i}$ is at most the largest root of
  $f_{\emptyset}$.
\end{lemma}
\begin{proof}
Let the polynomials be of degree $n$.
Let $g$ be a polynomial that interlaces all of the $f_{i}$, 
  and  let $\alpha_{n-1}$ be the largest root of $g$.
As each $f_{i}$ has a positive leading coefficient, it is positive for sufficiently large $x$.
As each $f_{i}$ has exactly one root that is at least $\alpha_{n-1}$, each
  $f_{i}$ is non-positive at $\alpha_{n-1}$.
So, $f_{\emptyset}$ is also non-positive at $\alpha_{n-1}$, and eventually
  becomes positive.
This tells us that $f_{\emptyset}$ has a root that is at least $\alpha_{n-1}$,
  and so its largest root is at least $\alpha_{n-1}$.
Let $\beta_{n}$ be this root.

As $f_{\emptyset}$ is the sum of the $f_{i}$, there must be some $i$
  for which $f_{i} (\beta_{n}) \geq 0$.
As $f_{i}$ has at most one root that is at least $\alpha_{n-1}$,
  and $f_{i} (\alpha_{n-1}) \leq 0$, the largest root of $f_{i}$  
  is it at least $\alpha_{n-1}$ and at most $\beta_{n}$.
\end{proof}
%\nikhil{should we mention that the proof actually shows that for every $k$ 
%$\lambda_k(f_\emptyset)\in \mathrm{conv}(f_s)$? This addresses one of the
%reviewer comments about what happens to the smallest root, etc.}

One can show that the assumptions of the lemma imply that $f_{\emptyset}$
  is itself a real-rooted polynomial. 
The conclusion of the lemma also holds for the $k$th largest root 
  by a similar argument.
However, we will not require these facts here.

If the polynomials do not have a common interlacing, the sum may fail to be real rooted:
 consider $(x+1) (x+2) + (x-1) (x-2)$.
Even if the sum of two polynomials is real rooted, the conclusion of Lemma~\ref{lem:interlacing}
  may fail to hold if the interval containing the largest roots of each polynomial overlaps
  the interval containing their second-largest roots.
For example,
  consider the sum of the polynomials $(x+5) (x-9) (x-10)$ and
  $(x+6) (x-1) (x-8)$.
It has roots at approximately $-5.3$, $6.4$, and $7.4$, so 
  its largest root is smaller than the largest root of both polynomials
  of which it is the sum.

\begin{definition}\label{def:family}
Let $S_{1}, \dots , S_{m}$ be finite sets and for every assignment $s_{1}, \dots, s_{m} \in S_{1} \times \dots \times S_{m}$
  let $f_{s_{1}, \dots , s_{m}} (x)$ be a real-rooted degree $n$ polynomial with positive leading coefficient.
For a partial assignment $s_1, \dots, s_k \in S_1 \times \ldots \times S_k$ with $k < m$, define
\[
  f_{s_{1},\dots , s_{k}} \defeq
\sum_{s_{k+1} \in S_{k+1}, \dots , s_{m} \in S_{m}}
  f_{s_{1}, \dots ,s_{k}, s_{k+1}, \dots , s_{m}},
\]
as well as
\[
  f_{\emptyset} \defeq \sum_{s_{1} \in S_{1}, \dots , s_{m} \in S_{m}}
  f_{s_{1}, \dots , s_{m}}.
\]

We say that the polynomials $\{f_{s_{1}, \dots , s_{m}} \}_{s_{1}, \dots , s_{m}}$
  form an \textit{interlacing family} if for all $k=0,\ldots, m-1$, and all
  $s_{1}, \dots, s_{k} \in S_{1} \times \dots \times S_{k}$,
  the polynomials
\[
  \{f_{s_{1}, \dots , s_{k},t}\}_{t\in S_{k+1}}
\]
have a common interlacing.
\end{definition}

\begin{theorem}\label{thm:interlacing}
Let $S_{1}, \dots , S_{m}$
  be finite sets and let
 $\setof{f_{s_{1}, \dots , s_{m}} }$ be an interlacing family of polynomials.
Then, there exists some $s_{1},\dots , s_{m} \in S_{1} \times \dots \times S_{m}$
  so that the largest root of
  $f_{s_{1}, \dots , s_{m}}$ is less than the largest root of $f_{\emptyset}$.
\end{theorem}
\begin{proof}
From the definition of an interlacing family, we know that the polynomials $\{ f_t\}$ for ${t \in S_1}$
  have a common interlacing and that their sum is $f_{\emptyset}$.
So, Lemma~\ref{lem:interlacing} tells us that one of the polynomials has largest root at most the largest root of $f_{\emptyset}$.
We now proceed inductively.
For any $s_{1}, \dots , s_{k}$, we know that the polynomials
  $\{f_{s_{1}, \dots , s_k, t}\}$ for $t \in S_{k+1}$ have a common interlacing and
  that their sum is $f_{s_{1}, \dots , s_{k}}$.
So, for some choice of $t$ (say $s_{k+1}$) the largest root of the polynomial
  $f_{s_{1}, \dots , s_{k+1}}$
  is at most the largest root of
  $f_{s_{1}, \dots , s_{k}}$.
\end{proof}

We will prove that the polynomials $\setof{f_{s}}_{s \in \setof{\pm 1}^{m}}$ defined in Section~\ref{sec:matching} are an interlacing family.
According to definition \ref{def:family}, this requires establishing the
  existence of certain common interlacings.
There is a systematic way to do this based on the fact that common interlacings
  are equivalent to real-rootedness statements.
In particular the following result seems to have been discovered a number of times.
It appears as Theorem~$2.1$ of Dedieu~\cite{Dedieu}, (essentially) as
Theorem~$2'$ of Fell~\cite{Fell}, and as (a special case of) Theorem 3.6 of Chudnovsky and
Seymour~\cite{ChudnovskySeymour}.
%In the case that the roots of $f$ and $g$ are distinct, it appears as Proposition~$1.35$
%  in Fisk~\cite{FiskBook}. 

\begin{lemma}\label{lem:Fisk}
Let $f_1,\ldots, f_k$ be (univariate) polynomials of the same degree with
positive leading coefficients. Then $f_1,\ldots, f_k$ have a common interlacing if
and only if $\sum_{i=1}^k\lambda_i f_i$ is real rooted for all convex
combinations $\lambda_i\ge 0, \sum_{i=1}^k\lambda_i=1$.
\end{lemma}

\section{The main result}\label{sec:main}

Our proof that the polynomials $\setof{f_{s}}_{s \in \setof{\pm 1}^{m}}$ form an interlacing family
  relies on the following generalization of the fact that the 
  matching polynomial is real-rooted.
It amounts to saying that if we pick each sign independently
  with any probabilities,
  then the resulting polynomial is still real-rooted.

\begin{theorem}\label{thm:skewMatchingPolynomial}
Let $p_{1}, \dotsc , p_{m}$ be numbers in $[0,1]$.
Then, the following polynomial is real-rooted
\[
  \sum_{s \in \setof{\pm 1}^{m}} 
  \left(\prod_{i : s_{i} = 1} p_{i} \right)
  \left(\prod_{i : s_{i} = -1} (1 - p_{i}) \right) 
  f_{s} (x).
\]
\end{theorem}

We will prove this theorem using machinery that we develop in Section~\ref{sec:real_stability}. 
It immediately implies our main technical result as follows.

\begin{theorem}\label{thm:interlacingFamily}
The polynomials $\setof{f_{s}}_{s \in \setof{\pm 1}^{m}}$
  are an interlacing family.
\end{theorem}
\begin{proof}
We will show that for every
  $0 \leq k \leq m-1$,
  every partial assignment $s_{1} \in \pm 1, \dotsc , s_{k} \in \pm 1$,
  and every $\lambda \in [0,1]$,
 the polynomial
\[
  \lambda f_{s_{1}, \dotsc , s_{k}, 1} (x) + 
  (1-\lambda ) f_{s_{1}, \dotsc , s_{k}, -1} (x)
\]
is real-rooted.
The theorem will then follow from Lemma~\ref{lem:Fisk}.

To show that the above polynomial is real-rooted, we
  apply Theorem~\ref{thm:skewMatchingPolynomial} 
  with $p_{k+1} = \lambda$, 
  $p_{k+2}, \dotsc , p_{m} = 1/2$,
  and $p_{i} = (1+s_{i})/2$ for $1 \leq i \leq k$.
\end{proof}

\begin{theorem}\label{thm:goodLift}
Let $G$ be a graph with adjacency matrix $A$ and universal cover $T$. Then there
is a signing $s$ of $A$ so that all of the eigenvalues of $A_s$ are at most
$\rho(T)$. 
In particular, if $G$ is $d$-regular, there is a signing $s$ so that the eigenvalues of $A_s$ are at most
$2\sqrt{d-1}$.
\end{theorem}
\begin{proof}
The first statement follows immediately from Theorems
 \ref{thm:interlacing} and \ref{thm:interlacingFamily} and Lemma~\ref{lem:coverBound}.
The second statement follows by noting that the universal cover of a $d$-regular
 graph is the infinite $d$-regular tree, which has spectral radius at most $2\sqrt{d-1}$,
 or by directly appealing to Theorem~\ref{thm:matchingBound}.
\end{proof}

\begin{lemma}\label{lem:completeBipartite}
Every non-trivial eigenvalue of a complete $(c,d)$-biregular graph
  is zero.
\end{lemma}
\begin{proof}
The adjacency matrix of this graph has rank $2$, so all its
  eigenvalues other than $\pm \sqrt{cd}$
  must be zero.
\end{proof}

\begin{theorem}\label{thm:Ramanujan}
For every $d \geq 3$ there is an infinite sequence of 
  $d$-regular bipartite Ramanujan graphs.
\end{theorem}
\begin{proof}
We know from Lemma~\ref{lem:completeBipartite} that the
   complete bipartite graph of degree $d$ is Ramanujan.
By Lemma 3.1 of \cite{BiluLinial}
  and
 Theorem~\ref{thm:goodLift},
  for every $d$-regular bipartite Ramanujan graph $G$,
  there is a 2-lift 
  in which every non-trivial eigenvalue
  is at most $2 \sqrt{d-1}$.
As the
  2-lift of a bipartite graph is bipartite,
  and the eigenvalues of a bipartite graph
  are symmetric about $0$,
  this 2-lift is also a regular bipartite Ramanujan graph.

Thus, for every $d$-regular bipartite Ramanujan graph $G$,
  there is another $d$-regular bipartite Ramanujan graph
  with twice as many vertices.
\end{proof}

\begin{theorem}\label{thm:Irreg}
For every $c,d\ge 3$, there is an infinite sequence of $(c,d)$-biregular
  bipartite Ramanujan graphs.
\end{theorem}
\begin{proof}
We know from Lemma~\ref{lem:completeBipartite} that the
   complete $(c,d)$-biregular is Ramanujan.
We will use this as a base for a construction of an infinite sequence of $(c,d)$-biregular bipartite Ramanujan graphs.
Let $G$ be any $(c,d)$-biregular bipartite Ramanujan graph.
As mentioned in Section \ref{intro:irreg}, the universal cover of $G$ is the infinite $(c,d)$-biregular tree,
  which has spectral radius $\sqrt{c-1}+\sqrt{d-1}$.
Thus, Theorem \ref{thm:goodLift} tells us that there is a 2-lift of $G$ with all new
  eigenvalues at most $\sqrt{c-1}+\sqrt{d-1}$.
As this graph is bipartite, all of its non-trivial eigenvalues
  have absolute value at most $\sqrt{c-1}+\sqrt{d-1}$.
So, the resulting $2$-lift is a larger $(c,d)$-biregular bipartite
  Ramanujan graph.
\end{proof}

To conclude the section, we remark that repeated application of Theorem~\ref{thm:goodLift} can be used to 
  generate an infinite sequence of irregular Ramanujan graphs from {\em any} finite irregular bipartite Ramanujan graph, since
  all of the lifts produced will have 
  %(by definition, since they are connected) % they aren't necessarily connected -- the all + one is not
  the same universal cover.
In contrast, Lubotzky and Nagnibeda~\cite{lubotzky1998not} have shown that there exist infinite trees that cover infinitely many finite graphs but
  such that none of the finite graphs are Ramanujan.

\section{Real stable polynomials} \label{sec:real_stability}

In this section we will establish the real-rootedness of a class of polynomials which
  includes the polynomials of Theorem~\ref{thm:skewMatchingPolynomial}.
We will do this by considering a multivariate
generalization of real-rootedness called {\em real stability} (see, e.g., the surveys
\cite{pemantle, wagner}).
In particular, we will show that the univariate polynomials we are interested in are the images,
under a well-behaved linear transformation, of a multivariate real stable polynomial.
%and then appeal to results in the field of real stable polynomials to argue the real-rootedness of the original class.

\begin{definition}
A multivariate polynomial $f \in \R[z_1, \dots, z_n]$ is called {\em real stable} if it is the zero polynomial
 or if
\[f(z_1, \dots, z_n) \neq 0\]
whenever the imaginary part of every $z_{i}$ is strictly positive.
\end{definition}
Note that a real stable polynomial has real coefficients, but may be evaluated on
complex inputs.

We begin by considering certain determinantal polynomials whose real stability is guaranteed
  by the following lemma, which may be found in 
  Borcea and Br\"{a}nd\'{e}n \cite[Proposition 2.4]{borcea2008applications}.
\begin{lemma}\label{matrixPencils}
Let $A_1, \dots, A_m$ be positive semidefinite matrices.  Then
\[
\mydet{z_1 A_1 + \dots + z_m A_m}
\]
is real stable.
\end{lemma}
%
% Dan: I cut this because it is not obvious from the definition of stability that we give here.
% One needs the other definition.
%
%\adam{Added some stuff here} In the case that the $A_i$ are positive definite, the lemma easily follows from the observations that such matrices have invertible square roots and that their characteristic polynomials are real-rooted.
%The extension to positive semidefiniteness requires some extra continuity arguments, and so we refer the interested reader to \cite[Proposition 2.4]{borcea2008applications}.

%We will denote the collection of all such polynomials on $n$ variables by $\H_n$.
%\begin{remark}
%Borcea and Br\"{a}nd\'{e}n (Lemma~2.1 in \cite{BBWeylAlgebra}) note that there is an equivalent characterization of the spaces $\H_i$.
%Firstly we can define $\H_{0}$ to be the constant functions and $\H_1$ to consist of all (univariate) real-rooted polynomials.
%Then we can define $\H_n$ to consist of all multivariate polynomials $f(x_1, \dots, x_n)$ such that
%\[
%f(\vec{w} + \vec{v}t) \in \H_1
%\]
%for all $\vec{w} \in \R^n$ and all $\vec{v} \in \R^n$ with $v_i > 0$ for all $i$ (the equivalence can be seen via the bijection between $z_k$ and $w_k + v_k i$).
%\end{remark}
Real stable polynomials enjoy a number of useful closure properties. 
In particular, it is easy to see that if $f(x_1, \dots, x_k)$ and $g(y_{1}, \dots y_j)$ are real stable then $f(x_1, \dots, x_k) g(y_1, \dots, y_j)$ is real stable.
A standard limiting argument based on Hurwitz's theorem shows that the real stability of $f(x_1, \dots, x_k)$ implies the
real stability of $f(x_1, \dots, x_{k-1}, c)$ for every  $c \in \R$ (see, e.g., Lemma~2.4 in \cite{wagner}).
For a variable $x_{i}$, we let $Z_{x_{i}}$ be the operator on polynomials induced by setting this variable to zero.

In \cite{BBWeylAlgebra}, Borcea and Br\"{a}nd\'{e}n characterize an entire class of differential operators that preserve real stability.
To simplify notation, we will let $\partial_{z_{i}}$ denote the operation of partial differentiation
  with respect to $z_{i}$.
For $\alpha , \beta \in \mathbb{N}^n$, we use the notation
\[
  z^{\alpha} = \prod_{i=1}^{n} z_{i}^{\alpha_i}
\quad
\text{and}
\quad 
  \partial^{\beta} = \prod_{i=1}^{n} \left(\partial_{z_i}\right)^{\beta_{i}}.
\]

\begin{theorem}[Theorem 1.3 in \cite{BBWeylAlgebra}] \label{Weyl}
Let $T : \R[z_1, \dots , z_n ] \to  \R[z_1, \dots , z_n ]$ be an operator of the form
\[
T = \sum_{\alpha , \beta \in \mathbb{N}^n} c_{\alpha, \beta} z^\alpha \partial^{\beta}
\]
where 
 $c_{\alpha, \beta} \in \mathbb{R}$ and $c_{\alpha ,\beta}$
  is zero for all but finitely many terms.
Define
\[ 
F_T(z,w) := \sum_{\alpha, \beta} c_{\alpha, \beta} z^{\alpha}w^{\beta}.
\]
Then $T$ preserves real stability if and only if $F_T(z, -w)$ is real stable.
\end{theorem}

We will use a special case of this result.
\begin{corollary}\label{cor:Tpreserves}
For non-negative real numbers $p$ and $q$ and variables $u$ and $v$,
the operator $T = 1 + p \partial_{u} + q \partial_{v}$ preserves real stability.
\end{corollary}
\begin{proof}
We just need to show that the polynomial $1 - p u - q v$ is real stable.
To see this, consider $u$ and $v$ with positive imaginary parts.
The imaginary part of $1 - p u - q v$ will then be negative, and
  so cannot be zero.
\end{proof}

We now show how operators of the preceding kind can be used to generate the
  expected characteristic polynomials that appears in Theorem~\ref{thm:skewMatchingPolynomial}.
\begin{lemma}\label{lem:Toper}
For an invertible matrix $A$, vectors $a$ and $b$, and a number $p \in [0,1]$,
\begin{multline*}
Z_{u} Z_{v} (1 + p \partial_{u} + (1-p) \partial_{v})
\mydet{A + u a a^{T} + v b b^{T}} = 
p \mydet{A + a a^{T}} + (1-p) \mydet{A + b b^{T}}.
\end{multline*}
\end{lemma}
\begin{proof}
The matrix determinant lemma (see, e.g., \cite{harville2008matrix}) states that for every nonsingular matrix $A$ and every real number $t$,
\[
  \mydet{A + t a a^{T}} = \mydet{A} (1 + t a^{T} A^{-1} a).
\]
One consequence of this is Jacobi's formula for the derivative of the determinant:
\[
  \partial_{t} \mydet{A + t a a^{T}} = \mydet{A} (a^{T} A^{-1} a).
\]
This formula implies that
\begin{multline*}
Z_{u} Z_{v} (1 + p \partial_{u} + (1-p) \partial_{v})
\mydet{A + u a a^{T} + v b b^{T}}
= \mydet{A}
\left(
1 +
p  (a^{T} A^{-1} a)
+
(1-p)  (b^{T} A^{-1} b)
 \right).
\end{multline*}
By the matrix determinant lemma, this equals
\[
  p \mydet{A + a a^{T}} + (1-p) \mydet{A + b b^{T}}.
\]
\end{proof}

Using these tools, we prove our main technical result on real-rootedness. 

\begin{theorem} \label{thm:realroots}
Let $a_{1}, \dotsc , a_{m}$ and $b_{1}, \dotsc , b_{m}$ be vectors in $\R^n$, 
and let $p_{1}, \dotsc , p_{m}$ be real numbers in $[0,1]$,
  and let $D$ be a positive semidefinite matrix.
Then every (univariate) polynomial of the form 
$$
  P(x) \defeq \sum_{S \subseteq [m]}
  \left( \prod_{i\in S}{p_{i}} \right)
  \left( \prod_{i \not \in S}{1-p_{i}} \right)
  \mydet{x I + D + \sum_{i \in S} a_{i} a_{i}^{T} + \sum_{i \not \in S} b_{i} b_{i}^{T}}
$$
is real-rooted.
\end{theorem}
\begin{proof}
Let $u_1, \dots, u_m$ and $v_1, \dots, v_m$ be formal variables and define
$$
Q(x, u_1,\ldots,u_m,v_1,\ldots,v_m) = \mydet{xI + D + \sum_{i} u_{i} a_{i} a_{i}^{T} + \sum_{i} v_{i} b_{i} b_{i}^{T}}.
$$
Lemma~\ref{matrixPencils} implies that $Q$ is real stable.

We claim that we can rewrite $P$ as 
\[
P (x) = 
\left(
  \prod_{i=1}^{m} 
Z_{u_{i}} Z_{v_{i}}
T_{i}
 \right) 
Q(x, u_1,\ldots,u_m,v_1,\ldots,v_m)
,
\]
where $T_{i} = 1 + p_{i} \partial_{u_{i}} + (1-p_{i}) \partial_{v_{i}}$.
To see this, we prove by induction on $k$ that
\[
\left(
  \prod_{i=1}^{k} Z_{u_{i}} Z_{v_{i}} T_{i}
 \right) 
Q(x, u_1,\ldots,u_m,v_1,\ldots,v_m)
\]
equals
\begin{multline*}
 \sum_{S \subseteq [k]}
  \left( \prod_{i\in S}{p_{i}} \right)
  \left( \prod_{i \in [k] \setminus S}{1-p_{i}} \right) 
  \mydet{x I + D + \sum_{i \in S} a_{i} a_{i}^{T} + \sum_{i \in [k] \setminus S} b_{i} b_{i}^{T}
  + \sum_{i > k} u_{i} a_{i} a_{i}^{T}
  + v_{i} b_{i} b_{i}^{T}
}.
\end{multline*}
The base case ($k = 0$) is trivially true, as it is the definition of $Q$.
The inductive step follows from Lemma~\ref{lem:Toper}.
The case $k = m$ is exactly the claimed identity.

Starting with $Q$ (a real stable polynomial) we can then apply Corollary~\ref{cor:Tpreserves} and the closure of real stable polynomials under the restrictions of variables to real constants to see that each of the polynomials above, including $P(x)$, is also real stable.
As $P(x)$ is real stable and has one variable, it is real-rooted. 
\end{proof}

Alternatively, one can prove Theorem~\ref{thm:realroots} 
  by observing that $P$ is a \textit{mixed characteristic polynomial}
  and then applying results of the second paper in this series~\cite{IF2}.

\begin{proof}[Proof of Theorem~\ref{thm:skewMatchingPolynomial}]
For each vertex $u$, let $d_{u}$ be its degree, and
  let $d = \max_{u} d_{u}$.
We need to prove that the polynomial
\[
  \sum_{s \in \setof{\pm 1}^{m}} 
  \left(\prod_{i : s_{i} = 1} p_{i} \right)
  \left(\prod_{i : s_{i} = -1} (1 - p_{i}) \right) 
\mydet{x I - A_{s}}
\]
is real-rooted.
This is equivalent to proving that the the following
  polynomial is real-rooted
\begin{equation}\label{eqn:thmSkewPoly}
  \sum_{s \in \setof{\pm 1}^{m}} 
  \left(\prod_{i : s_{i} = 1} p_{i} \right)
  \left(\prod_{i : s_{i} = -1} (1 - p_{i}) \right) 
\mydet{x I + d I - A_{s}},
\end{equation}
as their roots only differ by $d$.

We now observe that the matrix $d I - A_{s}$ is a signed
  Laplacian matrix of $G$ plus a nonnegative diagonal matrix.
For each edge $(u,v)$, define the rank $1$-matrices
\begin{align*}
  L^{1}_{u,v} & = (e_{u} - e_{v}) (e_{u} - e_{v})^{T}, \quad  \text{and}\\
  L^{-1}_{u,v} & = (e_{u} + e_{v}) (e_{u} + e_{v})^{T},
\end{align*}
where $e_{u}$ is the elementary unit vector in direction $u$. 
Consider a signing $s$ and let $s_{u,v}$ denote the sign it assigns to edge $(u,v)$.
Since the original graph had maximum degree $d$, we have
\[
  d I - A_{s} = \sum_{(u,v) \in E} L^{{s}_{u,v}}_{u,v} + D,
\]
where $D$ is the diagonal matrix whose $u$th diagonal entry
  equals $d - d_{u}$.
As the diagonal entries of $D$ are non-negative, it is positive semidefinite.
If we now set $a_{u,v} = (e_{u} - e_{v})$
  and $b_{u,v} = (e_{u} + e_{v})$,
  we can express 
  the polynomial in \eqref{eqn:thmSkewPoly} as
\begin{multline*}
  \sum_{s \in \setof{\pm 1}^{m}} 
  \left(\prod_{i : s_{i} = 1} p_{i} \right)
  \left(\prod_{i : s_{i} = -1} (1 - p_{i}) \right) 
\mydet{x I + D
  + \sum_{s_{u,v} = 1} a_{u,v} a_{u,v}^{T} 
  + \sum_{s_{u,v} = -1} b_{u,v} b_{u,v}^{T}  
}.
\end{multline*}
The fact that this polynomial is real-rooted now follows from Theorem
  \ref{thm:realroots}.
\end{proof}

\section{Conclusion}

We conclude by drawing an analogy between our proof technique and the
probabilistic method, which relies on the fact
that for every random variable $X:\Omega \rightarrow \R$, there is 
an $\omega\in\Omega$ for which $X(\omega)\le \expec{}{X}$. 
We have shown that for certain special polynomial-valued random variables $P:\Omega\rightarrow \R[x]$, there must be an $\omega$ with $\lambda_{max}(P(\omega))\le
\lambda_{max}(\expec{}{P})$. 
In fact it is possible to define interlacing families in greater generality than we have done here, using probabilistic notation. 
In particular, we call a polynomial-valued random
variable $P$ {\em useful} if $P$ is deterministic and real-rooted {\em or} if there exist disjoint
non-trivial events $E_1,\ldots, E_k$ with $\sum_{i\le k} \prob{}{E_i}=1$ such
that the polynomials $\{\expec{}{P|E_i}\}_{i\le k}$ have a common interlacing and each
polynomial $\expec{}{P|E_i}$ is itself useful. 
The conclusion of Theorem~\ref{thm:interlacing} continues to hold for this definition, and we suspect it will be useful in non-product settings. 
In the case of this paper, the
events $E_i$ are particularly simple: they correspond to setting one sign of a
lift to be $+1$ or $-1$, and the resulting sequence of polynomials
$f_\emptyset,f_{s_1},\ldots,f_{s_1,\ldots,s_m}$ forms a martingale (a fact that we do not use, but may be interesting in its own right).

Like many applications of the probabilistic method, our proof does not yield a
polynomial-time algorithm. 
In the particular case of random lifts, the
polynomial $f_\emptyset$ is itself a matching polynomial, which is $\# P$-hard
to compute in general. 
It would certainly be interesting to find computationally efficient analogues of our method.

\section*{Acknowledgment}
This research was partially supported by NSF grants CCF-0915487 and  CCF-1111257,
  an NSF Mathematical Sciences Postdoctoral Research Fellowship, Grant No. DMS-0902962,
  a Simons Investigator Award, and a MacArthur Fellowship.

We thank James Lee for suggesting Lemma~\ref{lem:Toper} and the simpler proof of Theorem \ref{thm:realroots} that appears here.
We thank Mirk{\'o} Visontai for bringing references \cite{GodsilGutman}, \cite{Fell}, 
  \cite{ChudnovskySeymour}, and \cite{Dedieu} to our attention.



%\IEEEtriggeratref{14}
%\IEEEtriggeratref{28}


\bibliographystyle{abbrv}
\bibliography{lifts}

\appendix

\section{Proof of Theorem~\ref{thm:matching}}\label{sec:ggtheorem}

%\begin{proof}[Proof of Theorem~\ref{thm:matching}]
Let $\sym{S}$ denote the set of permutations of a set $S$ and let $|\pi|$ denote the number of inversions of a permutation $\pi$.
Expanding the determinant as a sum over permutations $\sigma \in \sym{[n]}$, we have
\begin{align*}
\lefteqn{\expec{s}{\det(xI-A_s)}} \\
&=\expec{s}{\sum_{\sigma \in \sym{[n]}} (-1)^{|\sigma|} \prod_{i=1}^n(xI-A_s)_{i, \sigma(i)}}
\\&=\sum_{k=0}^n x^{n-k}\sum_{S\subset [n],|S|=k}\sum_{\pi \in
\sym{S}}\expec{s}{(-1)^{|\pi|} \prod_{i\in S} (A_s)_{i, \pi(i)}}
\\&\qquad\textrm{where $\pi$ denotes the part of $\sigma$ with $\sigma(i)\neq i$}
\\&=\sum_{k=0}^n x^{n-k}\sum_{S\subset
[n],|S|=k}\sum_{\pi\in\sym{S}}\expec{s}{(-1)^{|\pi|} \prod_{i\in S} s_{i, \pi(i)}}.
\end{align*}
Since the $s_{ij}$ are independent with $\expec{}{s_{ij}}=0$, only those products which contain even powers ($0$ or $2$) 
of the $s_{ij}$ survive. Thus, we may restrict our attention to the permutations $\pi$ which contain only orbits of size two.
These are just the perfect matchings on $S$.
There are no perfect matchings when $\sizeof{S}$ is odd; 
otherwise, each matching consists of 
$|S|/2$ inversions.
Since $\expec{s}{s_{ij}^2} = 1$, we are left with

\begin{align*}
\lefteqn{\expec{s}{\det(xI-A_s)}} \\
& =\sum_{k=0}^n
x^{n-k}\sum_{|S|=k}\sum_{\textrm{matching $\pi$ on $S$}} (-1)^{|S|/2}\cdot 1 
\\
&= \mu_G(x),
\end{align*}
as desired.
%\end{proof}

\end{document}